\documentclass{article}
\usepackage{amsthm,amsmath,amssymb}
\usepackage{amsthm,amsmath,amssymb}
\usepackage{amsfonts}
\usepackage{amssymb}
\usepackage{graphicx}
\usepackage{amsmath}
\usepackage[a4paper]{geometry}
\usepackage{color}
\usepackage{amsfonts, amsmath, amssymb, amsthm, constants}

\usepackage{dsfont}
\usepackage{mathrsfs}
\usepackage{verbatim}

\usepackage{url}
\usepackage{color}
\usepackage{amsfonts}

\newtheorem{theorem}{Theorem}

\newtheorem{corollary}{Corollary}

\newtheorem{definition}{Definition}

\newtheorem{lemma}{Lemma}

\newcommand{\pp}{\mathbb P}

\definecolor{jens}{rgb}{0,0.4,0.8}

\def\bR{\mathbb R}
\def\bE{\mathbb E}

\newtheorem{Theorem}{Theorem}
\newtheorem{Assumption}[Theorem]{Assumption}

\def\argmin{\mathop{\rm arg\,min}}

\def\rank{{\rm rank}}

\def\Var{\mathop{\rm Var}\nolimits}
\def\ba{\mathbf{a}}

\begin{document}
\title{Constructing confidence sets for the matrix completion problem}

\maketitle

\author{\begin{center}
		Alexandra Carpentier, \textit{ Universit\"at Potsdam\footnote{Institut f\"ur Mathematik,  carpentier@maths.uni-potsdam.de}}\\ \vskip 0.25 cm Olga Klopp, \textit{ University Paris Nanterre\footnote{MODAL'X,  kloppolga@math.cnrs.fr}}\\ \vskip 0.25 cm Matthias L\"offler,  \textit{ University of Cambridge\footnote{Statistical Laboratory, 
				Centre for Mathematical Sciences, m.loffler@statslab.cam.ac.uk}}\end{center} }

\begin{abstract}
	In the present note we consider the problem of constructing honest and adaptive confidence sets 
	for the matrix completion problem. For the Bernoulli model  with known variance of the noise we provide a  realizable method for constructing 	 confidence sets that adapt to the unknown rank of the true matrix.
\end{abstract}
\section{Introduction}

\noindent
In recent years, there has been a considerable interest in statistical inference for high-dimensional matrices.   One particular problem is matrix completion where one observes only a small number $n \ll m_1m_2$ of the entries of a high-dimensional $m_1\times m_2$ matrix $M_0$ of unknown rank $r$; it aims at inferring the missing entries.   The problem of matrix completion  comes up in many areas including collaborative filtering, multi-class learning in data analysis, system identification in control, global positioning from partial distance information and computer vision, to mention some of them. For instance, in computer vision, this problem arises as many pixels may be missing in digital images. In collaborative filtering, one wants to make automatic predictions about the preferences of a user by collecting information from many users. So, we have a data matrix where rows are users and columns are items. For each user, we have a partial list of his preferences. We would like to predict the missing ones in order to be able to recommend items that he may be interested in.

In general, recovery of a matrix from a small number of observed entries
is impossible, but, if the unknown matrix has low rank, then accurate and even exact recovery is possible. In the  noiseless setting, \cite{Candes_Tao,candes_recht,candes_plan}
established the following remarkable result: assuming that it satisfies a low coherence condition, $M_0$  can be recovered exactly by constrained nuclear norm minimization with high probability from only $n \gtrsim  r (m_1\vee m_2)\log^2(m_1\vee m_2)$ entries observed uniformly at random. 

What makes low-rank matrices special is that they depend on a number of free parameters that is much smaller than the total number of entries. Taking the singular value decomposition of a matrix $A\in \mathbb{R}^{m_1\times m_2}$ of rank $r$, it is easy to see that $A$ depends upon $(m_1+m_2)r-r^{2}$ free parameters. This number of free parameters gives us a lower bound for the number of observations needed to complete the matrix. 

A situation, common  in applications, corresponds to
the noisy setting in which the few available entries are corrupted by noise.  Noisy matrix completion has
been extensively studied recently (e.g.,  \cite{Koltchinskii_Lounici_Tsybakov,Negahban_Wainwright, cai_max, Chatterjee_mc}). 
 Here we observe a relatively small number of entries of a data matrix \begin{equation*} 
Y=M_0+E
\end{equation*}
where $M_0=(M_{ij})\in \mathbb{R}^{m_{1}\times m_{2}}$ is the unknown matrix of interest and $E = (\varepsilon_{ij}) \in \mathbb{R}^{m_1\times m_2}$ is a matrix of random errors.  
 It is an important issue in applications to be able to say from the observations how well the recovery procedure has worked or, in the sequential sampling setting, to be able to give data-driven stopping rules that guarantee the recovery of the matrix $M_0$ at a given  precision. This fundamental statistical question was recently studied in \cite{klopp_ci} where two statistical models for matrix completion are considered: the \emph{trace regression model}  and the \emph{Bernoulli model} (for details see Section \ref{sec:not}). In particular, in  \cite{klopp_ci}, the authors show that in the case of  unknown noise variance, the information-theoretic structure of these two models  is fundamentally different. In  the trace regression model, even if only an upper bound for the variance of the noise is known, a honest and rank adaptive Frobenius-confidence set  whose diameter scales with the  minimax optimal estimation rate exists. In the Bernoulli model however, such sets do not exist.
 
 Another major difference is that, in the case of known variance of the noise, \cite{klopp_ci} provides  a realizable method for constructing  confidence sets for the trace regression model whereas for the Bernoulli model only the existence of adaptive and honest confidence sets is demonstrated. The proof uses the duality between the problem of testing the rank of a matrix and the existence of honest and adaptive confidence sets. In particular, the construction in \cite{klopp_ci}  is based on infimum test statistics which can not be computed in polynomial time for the matrix completion problem. 
  The present note aims to close this gap and provides a realizable method for constructing  confidence sets for the Bernoulli model. 

\subsection{Notation, assumptions and some basic results}\label{sec:not}
We assume that each entry of $Y$ is observed independently of the other entries with probability $p=n / (m_1 m_2)$. More precisely, if $n \le m_1 m_2$ is given and $B_{ij}$ are i.i.d.~Bernoulli random variables  of parameter $p$ independent of the $\varepsilon_{ij}$'s, we observe
\begin{equation} \label{Bernoulli_model}
Y_{ij}=B_{ij}\left (M_{ij}+\varepsilon_{ij}\right),~~1 \le i \le m_1, 1 \le j \le m_2.
\end{equation}
This model for the matrix completion problem is usually called the \textit{Bernoulli model}.
Another model often considered in the matrix completion literature is the trace regression model (e.g.,  \cite{Koltchinskii_Lounici_Tsybakov,Negahban_Wainwright, cai_max,klopp_general}). Let $k_0=\rank (M_0)$.

In many of the most cited applications of the matrix completion problem, such as recommendation systems or the problem of global positioning from the local distances, the noise is bounded but not necessarily identically distributed. This is the assumption  which we adopt in the present paper. More precisely, we assume that the noise variables are independent, homoscedastic, bounded and centered:
\begin{Assumption}\label{noise_boundedTR}
	For any $(ij)\in [m_1]\times [m_2]$ we assume that	$\bE(\varepsilon_{ij})=0$, $\bE(\varepsilon_{{ij}}^{2})=\sigma^{2}$ and that there exists a positive constant  $U>0$ such that
	\begin{equation*} 
	\underset{{ij}}{\max}\left \vert\varepsilon_{{ij}}\right \vert\leq U.
	\end{equation*}
\end{Assumption} \noindent
%

Let $m=\min(m_1,m_2)$, $d=m_1+m_2$. For any $l\in \mathbb{N}$ we set  $[l]=\{1,\dots,l \}$. For any integer $0\leq k\leq m$ and any $\ba>0$, we  define the parameter space of rank $k$ matrices with entries bounded by $\ba$ in absolute value as 
\begin{equation}\label{class_matrices}
 \begin{split}
 {\cal A}(k,\ba)
 &= \left \{ M\in\,\mathbb R^{m_1\times m_2}:\,
 \mathrm{rank}(M)\leq k,\,\Vert M\Vert_{\infty}\leq \ba \right \}.
 \end{split}
 \end{equation}
 For constants $\beta \in (0,1)$ and $c=c(\sigma, \ba) > 0$ we have that
 \begin{equation*} 
 \inf_{\widehat M} \sup_{M_0 \in \mathcal{A}(k,\ba)} \mathbb{P}_{M_0, \sigma} \left (   \dfrac{\|\widehat M- M_0\|^2_2}{m_1m_2} > c\frac{
 	kd }{n} \right ) \geq \beta
 \end{equation*}
where $\widehat M$ is an estimator of $M_0$ (see, e.g.,  \cite{klopp_thresholding}).
 It has been also shown in \cite{klopp_thresholding} that an iterative soft thresholding estimator $\widehat M $ satisfies with $\mathbb{P}_{M_0, \sigma}$-probability at least $1-8/d$
 \begin{equation} \label{upper_bound_bernoulli}
 \dfrac{\Vert \widehat M-M_0\Vert_{F}^{2}}{m_1m_2}\leq C\frac{
 	(\ba+\sigma)^2
 	kd}{n} \quad \text{and} \quad \Vert M_0 - \widehat M\Vert_{\infty}\leq 2\ba
 \end{equation} 
 for a constant $C> 0$. 
 These lower and upper bounds imply that for the Frobenius loss 
%
 the  minimax risk for recovering a  matrix $M_0\in {\cal A}(k_0,\ba)$ is of order $\sqrt{\dfrac{(\sigma+\ba)^{2} k_0dm_1m_2}{n}}$.
\\For $k\in[m]$ we set $$r_k= C\frac{(\sigma+\ba)^{2} dk}{n},$$
where $C$ is the numerical constant in \eqref{upper_bound_bernoulli}.


Let $A,B$ be matrices in $\mathbb{R}^{m_{1}\times m_{2}}$.
We define the \textit{matrix scalar product} as
$\langle A,B\rangle =\mathrm{tr}(A^{T}B)$.
The trace norm of a matrix $A=(a_{ij})$ is defined as  $\Vert A\Vert_{*}:=\sum \sigma_j(A)$, the operator norm as
$\Vert A\Vert:=\sigma_1(A)$ and the Frobenius norm as $\|A \|_2^2:=\sum_i \sigma_i^2 = \sum_{i,j} a_{ij}^2$ 
where  $(\sigma_j(A))_j$ are the singular values of $A$ ordered decreasingly.
$\left\Vert A\right\Vert_{\infty}=\underset{i,j}{\max}|a_{ij}|$  denotes the largest absolute value of any entry of $A$. 

In what follows, we use  symbols $C,c$ for a generic positive
constant, which is independent of $n$, $m_1,m_2$, and may take different
values at different places. We denote by $a\vee b=\max(a,b)$.

We use the following definition of honest and adaptive confidence sets: 
\begin{definition} 
	Let $\alpha, \alpha'>0$ be given. A set  $C_n=C_n((Y_{ij},B_{ij}), \alpha) \subset \mathcal{A}(m,\ba)$ is a honest confidence set at level $\alpha$ for the model $\mathcal A(m,\ba)$ if
	\begin{equation*} 
	 \liminf_{n}\inf_{M \in \mathcal A(m,\ba)} \mathbb P^n_{M}(M \in C_n) \geq 1-\alpha. \end{equation*}
	Furthermore, we say that $C_n$ is adaptive for the sub-model  $\mathcal{A}(k,\ba)$ at level $\alpha'$ if there exists a constant $C=C(\alpha, \alpha') > 0$ such that
	\begin{equation*} 
	\sup_{M\in \mathcal{A}(k,\ba)} \mathbb{P}^n_{M} \left (\|C_n\|_{2} > C r_k \right ) \leq \alpha'\end{equation*}
	while still retaining
	\begin{equation*} 
	\sup_{M\in \mathcal{A}(m,\ba)} \mathbb{P}^n_{M} \left (\|C_n\|_{2} > C r_m \right ) \leq \alpha'.\end{equation*}
\end{definition}

\section{A non-asymptotic confidence set for matrix completion problem}

 Let $\widehat M$ be an estimator of $M_0$ based on the observations $(Y_{ij},B_{ij})$ from the Bernoulli model \eqref{Bernoulli_model} such that $\Vert \widehat M\Vert_{\infty}\leq \ba$. Assume that for some $\beta > 0$  $\widehat M$ satisfies  the following risk bound: 
\begin{equation} \label{upper_bound_thresholding}
  \sup_{M_{0} \in \mathcal{A}(k_0,\ba)}  \mathbb{P}\left (\dfrac{\|\widehat M- M_0\|^2_2}{m_1m_2} \leq C
   \frac{(\sigma+\ba)^{2} k_0d}{n}\right )\geq 1-\beta.
 \end{equation}
We can take, for example, the thresholding, estimator considered in  \cite{klopp_thresholding} which attains \eqref{upper_bound_thresholding} with $\beta=8/d.$
  Our construction is based on Lepski's method. We denote by $\widehat M_{k}$ the projection of $\widehat M$ on  the set $\mathcal{A}(k,\ba)$ of matrices of rank $k$ with sup-norm bounded by $\ba$:
\begin{equation*}
\widehat M_{k}\in \underset{A\in \mathcal{A}(k,\ba) }{\argmin}\Vert \widehat M-A\Vert_{2}.
\end{equation*}
Set $$S=\{k:\;\Vert \widehat M-\widehat M_k\Vert^{2}_{2}\leq r_k\}\qquad \text{ and}\qquad k^{*}=\min \{k\in S\}.$$ We will use $\widetilde M=\widehat M_{k^{*}}$ to center our confidence set and the residual sum of squares statistic $\hat r_n$:
 \begin{equation}\label{residual}
 \hat r_n = \frac{1}{n} \sum_{ij}(Y_{ij}-B_{ij}\widetilde M_{ij})^2 - \sigma^2.
 \end{equation}
Given $\alpha>0$, let $$\bar z = \dfrac{p}{256}\|M-\widetilde M\|^2_2+ z(UC^{*})^{2}dk^{*}\quad \text{and}\quad \xi_{\alpha, U}=2U^{2}\sqrt{\log(\alpha^{-1})}+\dfrac{4U^{2}\log(\alpha^{-1})}{3\sqrt{n}}.$$ 
Here  $z$ is a sufficiently large numerical constant to be chosen later on and $C^{*}\geq 2$ is an universal constant in Corollary 3.12 \cite{Bandeira}. We  define the confidence set as follows:
\begin{equation} \label{RSSconf_bernoulli}
C_n = \left\{M \in \mathbb R^{m_1\times m_2}: \dfrac{\|M-\widetilde M\|_2^2}{ m_1m_2} \le  128 \left(\hat r_n + \frac{\ba^{2}zdk^{*}+\bar z}{n} + \frac{\xi_{\alpha, U}}{\sqrt n}\right)  \right\}.
\end{equation} 
\begin{theorem} \label{thm_bernoulli_known variance}
Let $\alpha>0$, $d>16$ and suppose that $\widehat M$ attains the risk bound \eqref{upper_bound_thresholding} with probability at least $1-\beta$. Let $C_n$ be given by (\ref{RSSconf_bernoulli}).
Assume that  $\Vert M_0\Vert_{\infty}\leq \ba$ and that Assumption~\ref{noise_boundedTR} is satisfied. Then, for every $n\geq m\log(d)$, we have
 \begin{equation}\label{thm_proba}
 \pp_{M_0}(M_{0} \in C_n) \ge 1-\alpha - \exp(-cd).
 \end{equation} 
 Moreover, with probability at least $1-\beta-\exp(-cd)$
 \begin{equation} \label{diameter_bernoulli}
 \frac{\Vert C_n\Vert^2_2}{m_1m_2}\leq C\frac{(\sigma+\ba)^{2} dk_0}{n}.
 \end{equation}
\end{theorem}
Theorem \ref{thm_bernoulli_known variance} implies that $C_n$ is an honest and adaptive confidence set:
\begin{corollary}
Let $\alpha>0$, $d>16$ and suppose that $\widehat M$ attains the risk bound \eqref{upper_bound_thresholding} with probability at least $1-\beta$. Let $C_n$ be given by (\ref{RSSconf_bernoulli}).
 Assume that Assumption~\ref{noise_boundedTR} is satisfied. Then, for $n\geq m\log(d)$,  $C_n$ is a $\alpha + \exp(-cd)$ 
 honest confidence set for the model $\mathcal A(m,\ba)$ and adapts to every sub-model $\mathcal{A} (k, \bf{a})$, $1 \leq k \leq m$, at level $\beta+\exp(-cd)$.
\end{corollary}



\begin{proof}[Proof of Theorem \ref{thm_bernoulli_known variance}]  We consider the following  sets 
\begin{equation*} 
\mathcal{C}(k,\ba)=\left \{A\in\mathbb{R}^{m_1\times m_2} \,:\, \left\Vert A\right\Vert_{\infty}\leq \ba, \left\Vert M_0-A\right\Vert_2^{2}\geq \dfrac{256(\ba\vee U)^{2}zd}{p}\;\text{and}\; \rank (A)\leq k\right \}
\end{equation*}
and write
\begin{equation}\label{def_C}
\mathcal{C}=\cup^{m}_{k=1}\mathcal{C}(k,\ba).
\end{equation}
When $\left\Vert M_0-\widetilde{M}\right\Vert_{2}^{2}\leq \dfrac{256(\ba\vee U)^{2}zd}{p}$
we have that $M_0\in C_n$. So, we only need to consider the case $\left\Vert M_0-\widetilde{M}\right\Vert_{2}^{2}\geq \dfrac{256(\ba\vee U)^{2}zd}{p}$. In this case  we have that $\widetilde{M}\in \mathcal{C}$.
We introduce  the observation operator $\mathcal X$ defined as follows, 
$$\mathcal X\,:\bR^{m_1\times m_2}\rightarrow \bR ^{m_1\times m_2} \qquad\text{with}\qquad \mathcal X(A)=(B_{ij}a_{ij})_{ij}.~~ $$ and set $\|A \|_{L_2(\Pi)}^2=\mathbb E\|\mathcal X(A)\|_2^2=p\|A \|^{2}_2$.  We can  decompose
\begin{align*}
\hat r_n = n^{-1}\|\mathcal X(\widetilde M - M_0)\|_{2}^2 + 2n^{-1} \langle \mathcal X(E),  M_0 -\widetilde M \rangle + n^{-1}\|\mathcal X(E)\|_{2}^2-\sigma^{2}. 
\end{align*}
Then we can bound the probability $\pp_{M_0} (M_0 \notin C_n)$
by the sum of the following probabilities:
$$I := \pp_{M_0}\left(\frac{\|\widetilde M - M_0\|_{L_2(\Pi)}^2}{128} >  \|\mathcal X(\widetilde M - M_0)\|_{2}^2 + z\ba^{2}dk^{*}  \right),$$
$$II := \pp_{M_0}\left(-  2 \langle \mathcal X(E),  M_0-\widetilde M \rangle >   \bar z\right), $$
$$III := \pp_{M_0} \left(- \|\mathcal X(E)\|_{2}^2+n\sigma^{2} >\sqrt{n} \xi_{\alpha,U}\right). $$
By Lemma \ref{lem_bernoulli_rip}, the first probability is bounded by $8\exp\left (-4d\right )$ for $z\geq (27C^{*})^{2}$.
For the second term  we use Lemma \ref{stoch} which implies that $II\leq \exp(-cd)$ for $z\geq 6240$.
 Finally, for the third term, Bernstein's inequality implies
 \begin{align*}
          \mathbb{P} \left\{ - \|\mathcal X(E)\|_{2}^2+n\sigma^{2} >  t \right\} \le \exp \left( -  \frac{t^2}{2\sigma^{2}nU^2+\frac{2}{3}U^{2}t} \right).
          \end{align*}
 Taking $t=2U^{2}\sqrt{n\log(\alpha^{-1})}+\frac{4}{3}U^{2}\log(\alpha^{-1})$ we get that $III\leq \alpha$ by definition of $\xi_{\alpha,U}$. This completes the proof of \eqref{thm_proba}. 
 
 To  prove \eqref{diameter_bernoulli}, using  Lemma \ref{lem_bernoulli_rip} and Lemma \ref{stoch}, we can bound the  square Frobenius norm diameter of our confidence set $C_n$ defined in \eqref{RSSconf_bernoulli} as follows:
 \begin{equation*} 
 \frac{\Vert C_n\Vert^2_2}{m_1m_2}\lesssim \frac{\|\widetilde M -M_{0}\|_2^2}{m_1m_2} + \left (r_{k^{*}} + \frac{\xi_{\alpha, U}}{\sqrt n}\right ).
 \end{equation*}
  This bound holds on an event of probability at least
  $1-\exp(-cd)$. Now we restrict to the event where $\widehat M$ attains the risk bound in \eqref{upper_bound_thresholding} which happens with probability at least $1-\beta$. On this event, $M_0\in \mathcal{A}(k_0,\ba)$ implies $\Vert \widehat M- \widehat M_{k_0}\Vert_{2}^2\leq r_{k_0}$. So,  $k_0\in S$ and $k^{*}\leq k_0$. Now, the  triangle inequality and $r_{k^{*}}\leq r_{k_0}$  imply that on the intersection of those two events we have that$$ \|\widetilde M -M_{0}\|_2^2\lesssim m_1m_2\left (r_{k_0}+r_{k^{*}}\right )\lesssim m_1m_2r_{k_0}.$$  This, together with the definition of $\xi_{\alpha,U}$ and  the condition $n\leq m_1m_2$, completes the proof of \eqref{diameter_bernoulli}.
\end{proof}


\appendix
\section{Technical Lemmas}

\begin{lemma}\label{lem_bernoulli_rip}
With probability larger then $1-8\exp\left (-4d\right )$ we have that
$$\underset{A\in\mathcal{C}}{\sup}\dfrac{\left \vert \left\Vert \mathcal X(M_0-A)\right\Vert_{2}-\Vert M_0-A\Vert _{L_2(\Pi)}\right \vert-\frac{7}{8}\Vert M_0-A\Vert _{L_2(\Pi)}}{\ba\sqrt{\rank(A)d}}\leq 27\,C^{*}$$
where $C^{*}$ is an universal numerical constant and $\mathcal{C}$ is defined in \eqref{def_C}.

\end{lemma}
\begin{proof}
We have that 
\begin{align}\label{lemma_bernoulli_rip_1}
&\mathbb{P} \left(\underset{A\in \mathcal{C}}{\sup}\dfrac{\left \vert \left\Vert \mathcal X(M_0-A)\right\Vert_{2}-\Vert M_0-A\Vert _{L_2(\Pi)}\right \vert-\frac{7}{8}\Vert M_0-A\Vert _{L_2(\Pi)}}{\ba\sqrt{\rank(A)d}}\geq 27\,C^{*}\right )\nonumber\\&\leq
\sum_{k=1}^{k_0}\underset{\mathrm I}{\underbrace{\mathbb{P} \left(\underset{A\in \mathcal{C}(k,\ba)}{\sup}\left \vert \left\Vert \mathcal X(M_0-A)\right\Vert_{2}-\Vert M_0-A\Vert _{L_2(\Pi)}\right \vert-\frac{7}{8}\Vert M_0-A\Vert _{L_2(\Pi)}\geq 27\,C^{*}\ba\sqrt{kd}\right )}}.
\end{align}
In order to upper bound $\mathrm I$, we use a standard peeling argument. Let $\alpha=7/6$ and $\nu^{2}=\dfrac{470\ba^{2}zd}{p}$. For $l\in\mathbb N$ set $$S_l=\left \{A\in \mathcal{C}(k,\ba)\,:\,\alpha^{l}\nu \leq \Vert A-M_0\Vert _2\leq \alpha^{l+1}\nu\right \}.$$
Then
\begin{align*}
\mathrm I&\leq \sum _{l=1}^{\infty} \mathbb{P} \left(\underset{A\in S_l}{\sup}\left \vert \left\Vert \mathcal X(M_0-A)\right\Vert_{2}-\Vert M_0-A\Vert _{L_2(\Pi)}\right \vert\geq 27\,C^{*}\ba\sqrt{kd}+ \frac{7}{8}\alpha^{l}\ba\sqrt{470zd}\right )\\&\leq \sum _{l=1}^{\infty} \underset{\mathrm {II}}{\underbrace{\mathbb{P} \left(\underset{A\in \mathcal{C}(k,\ba,\alpha^{l+1}\nu)}{\sup}\left \vert \left\Vert \mathcal X(M_0-A)\right\Vert_{2}-\Vert M_0-A\Vert _{L_2(\Pi)}\right \vert\geq 27\,C^{*}\ba\sqrt{kd}+ \frac{7}{8}\alpha^{l}\ba\sqrt{470zd}\right )}}
\end{align*}
where $\mathcal{C}(k,\ba,T)=\{A\in \mathcal{C}(k,\ba)\,:\,  \left\Vert M_0-A\right\Vert_2\leq T \}$. The following Lemma gives an upper bound on $\mathrm {II}$:
\begin{lemma}\label{lemma_sup}
 Consider the following set of matrices $$\mathcal{C}(k,\ba,T)=\left \{A\in\mathcal{C}(k,\ba) \,:\,  \left\Vert M_0-A\right\Vert_2\leq T \right \}
 $$ and set $$Z_T=\underset{A\in \mathcal{C}(k,\ba,T)}{\sup}\left \vert \left\Vert \mathcal X(M_0-A) \right\Vert_{2}-\Vert M_0-A\Vert _{L_2(\Pi)}\right \vert.$$ Then, we have that  \begin{equation*}
  \mathbb{P}\left ( Z_{T} \geq \frac{3}{4}\sqrt{p}T+27\,C^{*}\ba\sqrt{kd}\right )\leq 4e^{-c_1\,p\,T^{2}/\ba^{2}}
  \end{equation*}
  with $c_1\geq (512)^{-1}$.
 \end{lemma}

 Lemma \ref{lemma_sup} implies that $\mathrm {II} \leq 8\exp(-c_1\,p\,\alpha^{2l}\nu^{2}/\ba^{2})$ and we obtain 
 \begin{equation*}
 \mathrm {I}\leq 8\sum_{l=1}^\infty\exp(-c_1\,p\,\alpha^{2l+2}\nu^{2}/\ba^{2})\leq 8\sum_{l=1}^\infty\exp\left (-2c_1\,p\,\nu^{2}\,\log(\alpha)\,l/\ba^{2}\right )
 \end{equation*} 
 where we used $e^{x}\geq x$. We finally compute for $\nu^{2}=188\ba^{2}zdp^{-1}$
 \begin{equation*}
 \mathrm {I}\leq \dfrac{8\exp\left (-2c_1\,p\,\nu^{2}\,\log(\alpha)/\ba^{2}\right )}{1-\exp\left (-2c_1\,p\,\nu^{2}\,\log(\alpha)/\ba^{2}\right )}\leq 16\exp\left (-376\,c_1zd\log(7/6)\right )\leq \exp(-5d)
 \end{equation*} 
 where we take $z\geq (27C^{*})^{2}$. Using \eqref{lemma_bernoulli_rip_1} and $d\geq \log(m)$ we get the statement of Lemma \ref{lem_bernoulli_rip}.\\
  

\end{proof}



%

\appendix

\begin{proof}[Proof of Lemma \ref{lemma_sup}]
This proof is close to the proof of Theorem 1 in \cite{Negahban_Wainwright}. We start by applying the  discretization argument. Let $\{G_{\delta}^{1},\dots G_{\delta}^{N(\delta)}\}$ be a $\delta-$covering of $\mathcal{C}(k,\ba,T)$ given by Lemma \ref{lemma_entropy}. Then, for any $A\in \mathcal{C}(k,\ba,T)$ there exists some index $i\in \{1,\dots,N(\delta)\}$ and a matrix $\Delta$ with $\Vert \Delta\Vert_{2}\leq \delta$ such that $A=G_{\delta}^{i}+\Delta$. Using the triangle inequality we have
   \begin{align*}
    \left \vert \Vert M_0-A\Vert _{L_2(\Pi)}-\left\Vert \mathcal X(M_0-A) \right\Vert_{2}\right \vert\leq \left \vert \left\Vert \mathcal X(M_0-G_{\delta}^{i}) \right\Vert_{2}-\Vert M_0-G_{\delta}^{i}\Vert _{L_2(\Pi)}\right \vert+\left\Vert \mathcal X \Delta \right\Vert_{2}+\sqrt{p}\delta.
     \end{align*}  
 Lemma \ref{lemma_entropy} implies that $\Delta\in \mathcal{D}_{\delta}(2k,2\ba,2T)$ where $$\mathcal{D}_{\delta}(k,\ba,T)=\{A\in \mathbb{R}^{m_1\times m_2}\;:\Vert A\Vert_{\infty}\leq \ba,\,\Vert A\Vert_{2}\leq \delta\quad \text{and}\quad \Vert A\Vert_{*}\leq \sqrt{k}T\}.$$ Then,
 \begin{align*}
     Z_{T}\leq \underset{i=1,\dots,N(\delta)}{\max}\left \vert \left\Vert \mathcal X(M_0-G_{\delta}^{i}) \right\Vert_{2}-\Vert M_0-G_{\delta}^{i}\Vert _{L_2(\Pi)}\right \vert+\underset{\Delta\in \mathcal{D}_{\delta}(2k,2\ba,2T) }{\sup}\left\Vert \mathcal X \Delta \right\Vert_{2}+\sqrt{p}\delta.
      \end{align*}  
 
 Now we take $\delta=T/8$ and use Lemma \ref{lemma_max} and Lemma \ref{lemma_concentration} to get
 \begin{align*}
      Z_{T}\leq \sqrt{p}\delta+ 8\,C^{*}\ba\sqrt{kd}+19\ba\,C^{*}\sqrt{2kd}+\sqrt{p}T/2+\sqrt{p}\delta\leq 27\,C^{*}\ba\sqrt{kd}+6\sqrt{p}T/8.
       \end{align*}  
       with probability at least $1-8\exp\left (-\frac{pT^{2}}{512\ba^{2}}\right )$.
       \end{proof}

 \begin{lemma}\label{lemma_entropy}
 	Let $\delta=T/8$. There  exists a set of matrices $\{G_{\delta}^{1},\dots G_{\delta}^{N(\delta)}\} $ with $N_{\delta}\leq \left (\frac{18T}{\delta}\right )^{2(d+1)k}$ and such that 
 	\begin{itemize}
 		\item [(i)] For any $A\in \mathcal{C}(k,\ba,T)$ there exists a $G_{\delta}^{A}\in \{G_{\delta}^{1},\dots G_{\delta}^{N(\delta)}\}$ satisfying $$\Vert A-G_{\delta}^{A}\Vert_{2}\leq \delta\quad \text{and}\quad (A-G_{\delta}^{A})\in \mathcal{D}_{\delta}(2k,2\ba,2T).$$
 		\item [(ii)] Moreover, $\Vert G^{j}_{\delta}-M_0\Vert_{\infty}\leq 2\ba$ and $\Vert G^{j}_{\delta}-M_0\Vert_{2}\leq 2T$ for any $j=1,\dots,N_{\delta}$.
 	\end{itemize}
 \end{lemma} 
 \begin{proof}
 	We use the following result (see Lemma 3.1 in \cite{candes_plan_tight} and Lemma A.2 in \cite{wang_xu}):
 	\begin{lemma}\label{lemma_candes}
 		Let $S(k,T)=\{A\in \mathbb{R}^{m_1\times m_2}\;:\rank (A)\leq k\quad \text{and}\quad \Vert A\Vert_{2}\leq T\}$. Then, there exists an $\epsilon-$net $\bar S(k,T)$ for the Frobenius norm obeying
 		$$\left \vert \bar S(k,T)\right \vert\leq \left (9T/\epsilon\right )^{(m_1+m_2+1)k}. $$
 	\end{lemma}
 	Let $S_{M_0}(k,T)=\{A\in \mathbb{R}^{m_1\times m_2}\;:\rank (A)\leq k\quad \text{and}\quad \Vert A-M_0\Vert_{2}\leq T\}$ and take a $X_0\in \mathcal{C}(k,\ba,T)$. We have that $S_{M_0}(k,T)-X_0\subset S(2k,2T)$. Let $\bar S(2k,2T)$ be an $\delta-$net given by Lemma \ref{lemma_candes}. Then, for any $A\in S_{M_0}(k,T)$ there exists a $\bar G^{A}_{\delta}\in \bar S(2k,2T)$ such that $\Vert A-X_0-\bar G^{A}_{\delta}\Vert_{2}\leq \delta$. Let $G^{j}_{\delta}=\Pi(\bar G^{j}_{\delta})+X_0$ for $j=1,\dots, \left \vert \bar S(2k,2T)\right \vert$ where $\Pi$ is the projection operator under Frobenius norm into the set $\mathcal{D}(2k,2\ba,2T)=\{A\in \mathbb{R}^{m_1\times m_2}\;:\Vert A\Vert_{\infty}\leq 2\ba,\quad \text{and}\quad \Vert A\Vert_{*}\leq 2\sqrt{2k}T\}$. Note that as $\mathcal{D}(2k,2\ba,2T)$ is convex and closed, $\Pi$ is non-expansive in Frobenius norm. For any $A\in \mathcal{C}(k,\ba,T)\subset S_{M_0}(k,T)$, we have that $ A-X_0\in \mathcal{D}(2k,2\ba,2T)$  which implies
 	$$\Vert A-X_0-\Pi(\bar G^{A}_{\delta})\Vert_{2}=\Vert \Pi(A-X_0-\bar G^{A}_{\delta})\Vert_{2}\leq
 	\Vert A-\bar G^{A}_{\delta}-X_0)\Vert_{2} \leq \delta$$
 	and we have that $(A-\Pi(\bar G^{A}_{\delta})-X_0)\in \mathcal{D}_{\delta}(2k,2\ba,2T)$ which completes the proof of (i) of Lemma \ref{lemma_entropy}. To prove (ii), note that by the definition of $\Pi$ we have that $\Vert G^{j}_{\delta}-M_0\Vert_{\infty}=\Vert \Pi(\bar G^{j}_{\delta})+X_0-M_0\Vert_{\infty}= \Vert \Pi(\bar G^{j}_{\delta}+X_0-M_0)\Vert_{\infty}\leq 2a$ and $\Vert G^{j}_{\delta}-M_0\Vert_{2}\leq 2T$.\end{proof}

 \begin{lemma}\label{lemma_concentration}
 Let $\delta=T/8$ and assume that $n\geq m\log(m)$. We have that with probability at least $1-4\exp\left (-\dfrac{pT^{2}}{512\ba^{2}}\right )$
 $$\underset{\Delta\in \mathcal{D}_{\delta}(2k,2a,2T) }{\sup}\left\Vert \mathcal X \Delta \right\Vert_{2}\leq 19\ba\,C^{*}\sqrt{2kd}+\sqrt{p}T/2. $$
 \end{lemma}
 \begin{proof}
  Let $X_T=\underset{\Delta\in \mathcal{D}_{\delta}(2k,2\ba,2T) }{\sup}\left\Vert \mathcal X \Delta \right\Vert_{2}$. We use the following Talagrand's concentration inequality :
  \begin{Theorem}\label{talagrand}
  Suppose that $f\,:\,[-1,1]^{N}\rightarrow \mathbb{R}$ is a convex Lipschitz function with Lipschitz constant $L$. Let $\Xi_1,\dots \Xi_N$ be independent random variables taking value in $[-1,1]$. Let $Z:\,=f(\Xi_1,\dots, \Xi_n)$. Then for any $t\geq 0$,
  $$\mathbb{P}\left (\left \vert Z-\mathbb{E}(Z)\right \vert\geq 16L+t\right )\leq 4e^{-t^{2}/2L^{2}}.$$
  \end{Theorem}
  For a proof see \cite{talagrand1996} and \cite{Chatterjee_mc}. Let $f(x_{11},\dots,x_{m_1m_2}):\,= \underset{\Delta\in \mathcal{D}_{\delta}(2k,2\ba,2T)}{\sup}\sqrt{ \sum _{(i,j)} x^{2}_{ij}\Delta^{2}_{ij}}.$   It is easy to see that $f(x_{11},\dots,x_{m_1m_2})$ is a Lipschitz function with Lipschitz constant $L=2\ba$. Indeed,
   \begin{equation*}
   \begin{split}
   \left \vert f(x_{11},\dots,x_{m_1m_2})-f(z_{11},\dots,z_{m_1m_2})\right \vert& \\&\hskip -3 cm =\left \vert \underset{\Delta\in \mathcal{D}_{\delta}(2k,2\ba,2T)}{\sup}\sqrt{ \sum _{(i,j)} x^{2}_{ij}\Delta^{2}_{ij}}-\underset{\Delta\in \mathcal{D}_{\delta}(2k,2\ba,2T)}{\sup}\sqrt{ \sum _{(i,j)} z^{2}_{ij}\Delta^{2}_{ij}}\right \vert\\
    &\hskip -2.5 cm\leq  \underset{\Delta\in \mathcal{D}_{\delta}(2k,2\ba,2T)}{\sup}\sqrt{ \sum _{(i,j)} (x_{ij}-z_{ij})^{2}\Delta^{2}_{ij}}\leq 2\ba \Vert x-z\Vert_{2}
     \end{split}
   \end{equation*}
  where $x=(x_{11},\dots,x_{m_1m_2})$ and $z=(z_{11},\dots,z_{m_1m_2})$. Now, Theorem \ref{talagrand} implies
 \begin{align}\label{azuma-hoeffding}
 \mathbb{P}\left ( X_{T}\geq\mathbb{E}(X_{T})+32\ba+ t\right )\leq 4\exp\left (-\frac{t^{2}}{8\ba^{2}}\right ).
 \end{align}
 Next, we bound the expectation $\mathbb{E}(X_{T})$. Applying Jensen's inequality, a symmetrization argument and the Ledoux-Talagrand contraction inequality (see, e.g., \cite{koltchiskii-stflour}) we get 
 \begin{align*}
 \left (\mathbb{E}(X_{T})\right )^{2}&\leq \bE\left (\underset{\Delta\in \mathcal{D}_{\delta}(2k,2\ba,2T)}{\sup}\sum_{(i,j)}B_{ij}\Delta^{2}_{ij}\right)\leq \bE\left (\underset{\Delta\in \mathcal{D}_{\delta}(2k,2\ba,2T)}{\sup}\sum_{(i,j)}B_{ij}\Delta^{2}_{ij}-\mathbb E\left (B_{ij}\Delta^{2}_{ij}\right )\right)+p\delta^{2}\\&\hskip -0.5 cm \leq 2\,\bE\left (\underset{\Delta\in \mathcal{D}_{\delta}(2k,2\ba,2T)}{\sup}\left \vert\sum_{(i,j)}\epsilon_{ij}B_{ij}\Delta^{2}_{ij}\right \vert\right )+p\delta^{2}\leq 8\ba\,\bE\left (\underset{\Delta\in \mathcal{D}_{\delta}(2k,2\ba,2T)}{\sup}\left \vert\sum_{(i,j)}\epsilon_{ij}B_{ij}\Delta_{ij}\right \vert\right )+p\delta^{2}\\&=8\ba\,\bE\left (\underset{\Delta\in \mathcal{D}_{\delta}(2k,2\ba,2T)}{\sup}\left \vert \left\langle \Sigma_R,\Delta\right\rangle\right \vert\right )+p\delta^{2}\leq 16\,\ba\sqrt{2k}\,T\,\bE\left ( \left\Vert \Sigma_R\right\Vert\right )+p\delta^{2}
 \end{align*}
 where $\{\epsilon_{ij}\}$ is i.i.d. Rademacher sequence, $\Sigma_R=\sum_{(i,j)}B_{ij}\epsilon_{ij}X_{ij}$ with $X_{ij}=e_{i}(m_1)e^T_{j}(m_2)$ and $e_{k}(l)$ are the canonical basis vectors in $\mathbb R^l$.  Lemma 4 in \cite{klopp_thresholding} and $n\geq m\log(m)$ imply that 
 \begin{equation}\label{expectation_sigma}
  \bE \left\Vert \Sigma_{R}\right\Vert\leq C^{*}\sqrt{pd}
  \end{equation}
  where $C^{*}\geq 2$ is an universal numerical constant. Using \eqref{expectation_sigma}, $\sqrt{x+y}\leq \sqrt{x}+\sqrt{y}$, $2xy\leq x^{2}+y^{2}$ and $\delta=T/8$ we compute
  \begin{align*}
  \mathbb{E}(X_{T})\leq 4\left (\ba\,C^{*}\sqrt{2kpd}\,T\right )^{1/2}+\sqrt{p}\delta\leq 16\ba\,C^{*}\sqrt{2kd}+3\sqrt{p}T/8.
  \end{align*}
  Taking in \eqref{azuma-hoeffding} $t=\sqrt{p}T/8$ we get the statement of Lemma \ref{lemma_concentration}.
 \end{proof}

 \begin{lemma}\label{lemma_max}
Let $\delta=T/8$, $d> 16$ and $(G^{1}_{\delta}, \ldots, G^{N(\delta)}_{\delta})$ be the collection of matrices given by Lemma \ref{lemma_entropy}. We have that
 \begin{align*}
 \underset{k=1,\dots,N(\delta)}{\max}\left \vert \left\Vert \mathcal X(M_0-G^{k}_{\delta}) \right\Vert_{2}-\Vert M_0-G^{k}_{\delta}\Vert _{L_2(\Pi)}\right \vert \leq \sqrt{p}\delta+ 8\,C^{*}\ba\sqrt{kd}
 \end{align*}
with probability at least $1-4\exp\left (-\frac{p\delta^{2}}{8\ba^{2}}\right )$.
 \end{lemma} 
 \begin{proof}
For any fixed $A\in \mathbb{R}^{m_1\times m_2}$ satisfying $\Vert A\Vert_{\infty}\leq 2a$ we have that
\begin{equation*}
\| \mathcal{X}A\|_{2}=\sqrt{\sum_{ij}B_{ij}A^{2}_{ij}}=\underset{\Vert u\Vert_{2}=1}{\sup}\sum_{ij}\left (u_{ij}B_{ij}A_{ij}\right ).
\end{equation*}
Then we can apply Theorem \ref{talagrand} with $f(x_{11},\dots,x_{m_1m_2}):\,=\underset{\Vert u\Vert_{2}=1}{\sup}\sum_{ij}\left (u_{ij}x_{ij}\right )$ to get
\begin{equation} \label{eq:tail_bound}
  \mathbb{P} \left(\left|\| \mathcal{X}A\|_{2} -\bE\| \mathcal{X}A\|_{2}\right| >  t+32\ba \right) \le 4 \exp \left\{- \frac{t^{2}}{8\ba^{2}}  \right\}.
  \end{equation}
  On the other hand let $Z=\underset{\Vert u\Vert_{2}=1}{\sup}\sum_{ij}\left (u_{ij}B_{ij}A_{ij}\right )$. Applying Corollary 4.8 from \cite{Ledoux_book} we get that $\Var Z=\Vert A\Vert^{2}_{L_2(\Pi)}-\left (\bE\| \mathcal{X}A\|_{2}\right )^{2}\leq 16^{2}\ba^{2}$ which together with \eqref{eq:tail_bound} implies
  \begin{equation} \label{eq:tail_bound_1}
    \mathbb{P} \left(\left|\| \mathcal{X}A\|_{2} -\| A\|_{L_2(\Pi)}\right| >  t+48\ba \right) \le 4 \exp \left\{- \frac{t^{2}}{8\ba^{2}}  \right\}.
    \end{equation}
Now Lemma \ref{lemma_max} follows from  Lemma \ref{lemma_entropy}, \eqref{eq:tail_bound_1} with $t=\sqrt{p}\delta + 5C^{*}\ba \sqrt{kd}$ and the union bound.  

 \end{proof}
 %



 
  \begin{lemma}\label{stoch}
  We have that
  \begin{align*}
  \underset{A\in\mathcal{C}}{\sup}\dfrac{\left \vert\langle \mathcal X(E), A - M_0 \rangle\right \vert-\frac{1}{256}\Vert M_0-A\Vert^{2}_{L_2(\Pi)}}{d\,\rank(A)}\leq 6240(UC^{*})^{2}.
  \end{align*}
   with probability larger then $1-\exp(-cd)$ with $c\geq 0,003$
   \end{lemma}         
  \begin{proof}
 Following the lines of the proof of Lemma \ref{lem_bernoulli_rip} with $\alpha=\sqrt{65/64}$ and  $\nu^{2}=\dfrac{252(\ba \vee U)^{2}zd}{p}$ we get
 \begin{align*}
 &\mathbb{P} \left(\underset{A\in\mathcal{C}}{\sup}\dfrac{\left \vert\langle \mathcal X(E), A - M_0 \rangle\right \vert-\frac{1}{256}\Vert M_0-A\Vert^{2}_{L_2(\Pi)}}{d\,\rank(A)}\geq 6240(UC^{*})^{2}\right )\\&\leq
 \sum_{k=1}^{k_0}\sum _{l=1}^{\infty} \mathbb{P} \left(\underset{A\in\mathcal{C}(k,a,\alpha^{l+1}\nu)}{\sup}\left \vert\langle \mathcal X(E), A - M_0 \rangle\right \vert\geq 6240(UC^{*})^{2}dk+\frac{p\alpha^{2l}\nu^{2}}{256}\right )\\&\leq 4\sum_{k=1}^{k_0}\sum _{l=1}^{\infty} \exp \left( - \frac{c_2p\alpha^{2l+2}\nu^{2}}{(\ba\vee U)^{2}}\right)\leq \exp{(-cd)}
 \end{align*}
 where we use the following lemma:
 \begin{lemma}\label{lemma_stoch_sup}
  Consider the following set of matrices $$\mathcal{C}(k,a,T)=\left \{A\in\mathcal{C}(k,a) \,:\,  \left\Vert M_0-A\right\Vert_2\leq T \right \}
  $$ and set $$\widetilde Z_T=\underset{A\in \mathcal{C}(k,a,T)}{\sup}\left \vert\langle \mathcal X(E), A - M_0 \rangle\right \vert.$$ We have that  \begin{equation*}
   \mathbb{P}\left (\widetilde Z_{T} \geq 6240(C^{*}U)^{2}dk+pT^{2}/260\right )\leq 4\exp \left( - pT^{2}/c_2(\ba\vee U)^{2}\right)
   \end{equation*}
   with $c_2\leq 12(1560)^{2}$
  \end{lemma}
\end{proof}    
   \begin{proof}[Proof of Lemma \ref{lemma_stoch_sup}]
         Fix a $X_0\in \mathcal{C}(k,a,T)$. For any $A\in \mathcal{C}(k,a,T)$, we set $\Delta=A-X_0$ and we have that $\rank(\Delta)\leq 2k$ and $\Vert \Delta\Vert_{2}\leq 2T$. Then using $ \left \vert\langle \mathcal X(E), A - M_0 \rangle\right \vert\leq \left \vert\langle \mathcal X(E), X_0 - M_0 \rangle\right \vert+ \left \vert\langle \mathcal X(E), \Delta \rangle\right \vert$
                   we get
                             \begin{align*}
                                \widetilde Z_{T}\leq \left \vert\langle \mathcal X(E), X_0 - M_0 \rangle\right \vert+\underset{\Delta\in \mathcal{T}(2k,2a,2T) }{\sup}\left \vert\langle \mathcal X(E), \Delta \rangle\right \vert 
                                  \end{align*} 
                                  where
                                  $$\mathcal{T}(2k,2a,2T)=\{A\in \mathbb{R}^{m_1\times m_2}\;:\Vert A\Vert_{\infty}\leq 2a,\,\Vert A\Vert_{2}\leq 2T\quad \text{and}\quad \rank(A)\leq 2k\}.$$
     
        Bernstein's inequality and $\Vert X_0-M_0\Vert_{2}\leq T$ imply
         that 
         \begin{align*}
           \mathbb{P} \left\{ \left \vert\langle \mathcal X(E), X_0-M_0 \rangle\right \vert >  t \right\} \le 2\exp \left( -  \frac{t^2}{2\sigma^{2}pT^2+\frac{4}{3}U \ba t}  \right).
           \end{align*}
      Taking $t=pT^{2}/520$ we get   
       \begin{align}\label{eq_bernstein_1}
               \mathbb{P} \left\{ \left \vert\langle \mathcal X(E), X_0-M_0 \rangle\right \vert > pT^{2}/520 \right\} \le 2\exp \left( - \frac{pT^{2}}{c_2(a\vee U)^{2}}\right).
               \end{align}   
        On the other hand, Lemma \ref{lemma_stoch_concentration} implies that
         with probability at least $1-2\exp\left (-\dfrac{pT^{2}}{c_3(a\vee U)^{2}}\right )$
          \begin{equation*} 
          \underset{\Delta\in \mathcal{T}(2k,2a,2T) }{\sup}\left \vert\langle \mathcal X(E), \Delta \rangle\right \vert\leq 6240(C^{*}U)^{2}kd+pT^{2}/520
          \end{equation*} 
         
   which together with \eqref{eq_bernstein_1} implies  the statement of  Lemma \ref{lemma_stoch_sup}. 
   \end{proof}
         
 \begin{lemma}\label{lemma_stoch_concentration}
 Assume that $n\geq m\log(m)$. We have that with probability at least $1-2\exp\left (-\dfrac{pT^{2}}{(a\vee U)^{2}}\right )$
 $$\underset{\Delta\in \mathcal{T}(2k,2a,2T) }{\sup}\left \vert\langle \mathcal X(E), \Delta \rangle\right \vert\leq 6240(C^{*}U)^{2}kd+pT^{2}/520 $$
 where $c_3$ is a numerical constant.
 \end{lemma}
 \begin{proof}
 Let $\widetilde X_T=\underset{\Delta\in \mathcal{T}(2k,2a,2T) }{\sup}\left \vert\langle \mathcal X(E), \Delta \rangle\right \vert=\underset{\Delta\in \mathcal{T}(2k,2a,2T) }{\sup}\langle \mathcal X(E), \Delta \rangle$. First we bound the expectation $\mathbb{E}(\widetilde X_{T})$:
  \begin{align*}
  \mathbb{E}(\widetilde X_{T})&\leq \bE\left (\underset{\Delta\in \mathcal{T}(2k,2a,2T)}{\sup}\left \vert\sum_{(i,j)}\varepsilon_{ij}B_{ij}\Delta_{ij}\right \vert\right )=\bE\left (\underset{\Delta\in \mathcal{T}(2k,2a,2T)}{\sup}\left \vert \left\langle \Sigma,\Delta\right\rangle\right \vert\right )\\&\leq 2\sqrt{2k}\,T\,\bE\left ( \left\Vert \Sigma\right\Vert\right )
  \end{align*}
  where  $\Sigma=\sum_{(i,j)}B_{ij}\varepsilon_{ij}X_{ij}$ with $X_{ij}=e_{i}(m_1)e^T_{j}(m_2)$ and $e_{k}(l)$ are the canonical basis vectors in $\mathbb R^l$.  Using $n\geq m\log(m)$ Lemma 4 in \cite{klopp_thresholding} and Corollary 3.3 in \cite{Bandeira} imply that 
  \begin{equation}\label{expectation_sigma_stoch}
   \bE \left\Vert \Sigma\right\Vert\leq C^{*}U\sqrt{pd}.
   \end{equation}
   where $C^{*}\geq 2$ is an universal numerical constant. Using \eqref{expectation_sigma_stoch}  we get 
   \begin{align}\label{bound_expectation}
   \mathbb{E}(\widetilde X_{T})\leq 2C^{*}U\sqrt{2kpd}\,T\leq 3120(C^{*}U)^{2}kd+pT^{2}/1560.
   \end{align}
   Now we use Theorem 3.3.16 in \cite{nickl_book} (see also Theorem 8.1 in \cite{klopp_ci}) to obtain
   \begin{align}\label{azuma_hoeffding_stoch}
    \mathbb{P}\left (\widetilde X_{T}\geq\mathbb{E}(\widetilde X_{T})+ t\right )&\leq \exp\left (-\frac{t^{2}}{4U\ba\mathbb{E}(\widetilde X_{T})+4\sigma^{2}pT^{2}+9U\ba t }\right )\nonumber\\
   & \leq \exp\left (-\frac{t^{2}}{8\ba U^{2}C^{*}\sqrt{2kpd}T+4\sigma^{2}pT^{2}+9U\ba t }\right )
    \end{align}
   Taking in \eqref{azuma_hoeffding_stoch} $t=pT^{2}/1560+2 UC^{*}\sqrt{2kpd}T $, together with \eqref{bound_expectation}  we get the statement of Lemma \ref{lemma_stoch_concentration}.
%

 \end{proof}

 \section*{Acknowledgements}  The work of A.~Carpentier is supported by the DFG’s Emmy Noether grant MuSyAD (CA 1488/1-1). The work of O. Klopp was conducted as part of the project Labex MME-DII (ANR11-LBX-0023-01). The work of M. L\"offler was supported by the UK Engineering and Physical Sciences Research Council (EPSRC) grant EP/L016516/1 for the University of Cambridge Centre for Doctoral Training, the Cambridge Centre for Analysis, and the European Research Council (ERC) grant No. 647812.



\end{document}